\documentclass[12pt,oneside,reqno]{amsart}
\usepackage{amsfonts,amssymb,amscd,amsmath,latexsym,enumerate,verbatim,calc, quotmark, mathtools}
\usepackage[all]{xy}

\usepackage[pagebackref]{hyperref}
\hypersetup{
     colorlinks = true, linkcolor = red,
     citecolor   = blue,
     urlcolor    = blue,
}

\def\NZQ{\mathbb}               % the font for N,Z,Q,R,C

\def\ZZ{{\NZQ Z}}
\def\Z{{\NZQ Z}}

\def\v{{\bf v}}
\def\w{{\bf w}}
\def\u{{\bf u}}

\def\Um{{\operatorname{Um}}}          %Unimodular row   
\def\E{{\operatorname{E}}}            %Elementary matrix
\def\ESp{{\operatorname{ESp}}}          %Elementary symplectic 
\def\GL{{\operatorname{GL}}}          %General group   
\def\K{{\operatorname{K}}}            %K group  
\def\M{{\operatorname{M}}}            %Matrix 
\def\MSE{{\operatorname{MSE}}}        %Higher Mennicke symbol
\def\SL{{\operatorname{SL}}}          %Special 
\def\SK{{\operatorname{SK}}}          %Special K
\def\Sp{{\operatorname{Sp}}}          %Elementary Symplectic group
\def\Um{{\operatorname{Um}}}          %Unimodular row   
\def\vpmod{\!\!\!\pmod} 
            
\def\WMS{{\operatorname{WMS}}}     %Weak Mennicke symbol
       
\def\SK{{\operatorname{SK}}}

\newtheorem{theorem}{Theorem}[section]
\newtheorem{lemma}[theorem]{Lemma}
\newtheorem{corollary}[theorem]{Corollary}
\newtheorem{proposition}[theorem]{Proposition}
\newtheorem{remark}[theorem]{Remark}

\newtheorem{definition}[theorem]{Definition}

\newtheorem{question}[theorem]{Question}
\newtheorem*{acknowledgement}{Acknowledgement}

\begin{document}

\title{\small Optimal injective stability for the 
symplectic {\bf $\K_1\Sp$} group}

\author{{\bf Anjan Gupta}\\
Tata Institute of Fundamental Research, Mumbai }

\footnote{Correspondence author: Anjan Gupta; 
{\it email: anjan@math.tifr.res.in, agmath@gmail.com}}

\maketitle

\noindent {\it Mathematics Subject Classification 13C10, 13H05, 19B14, 19B99.}

\vskip3mm


\vskip3mm

\begin{abstract} 

If $R$ is a commutative ring, $I$ an ideal of $R$ and $\v, \w \in \Um_{2n}(R, I)$ then we show that $\v, \w$ are in the same orbit of elementary action if and only if they are in that of elementary symplectic action extending a result in \cite{cr}. We also show that if $A$ is a non-singular affine algebra of dimension $d$ over an algebraically closed field $k$ such that $d! A = A$, $d \equiv 2 \vpmod
4$ and $I$ an ideal of $A$, then $\Um_d(A, I) = e_1{\Sp}_d(A, I)$. As a consequence it is proved that if $A$ is a non-singular affine algebra of dimension $d$ over an algebraically closed field $k$ such that $(d + 1)!A = A$, $d \equiv 1 \vpmod 4$ and $I$ a principal ideal then
$\Sp_{d-1}(A, I) \cap {\ESp}_{d+1}(A, I) = {\ESp}_{d -1}(A, I)$. We give an example to show that the above result does not hold true for an affine algebra over a $C_2$ field and also show by an example
that the above stability estimate is optimal.
%It is shown that if $A$ is a non-singular affine algebra of dimension $d \geq 3, d !A = A$
%over an algebraically closed field $k$ 
%and $I$ a principal ideal then $\SL_d(A, I) \cap \E(A, I) = \E_d(A, I)$.
%It is also shown that . As a consequence it is
%proved 
\end{abstract}

%%%%%%%%%%%%%%%%%%%%%%%%%%%%%%%%%%%%%%%%%%%%%%%%%%%%%%%%%%%%%%%%%%%%%%%%%%
\section{Introduction}

We shall confine ourselves to working over a commutative ring $R$ with
$1$, which is generally an affine algebra over an algebraically closed
field $k$.

Historically, Bass--Milnor--Serre in \cite{bms} observed that
$K_1(R) = \lim GL_n(R)/E_n(R)
$ is a finite limit, attained when $n \geq d + 3$, where $d$ is the Krull
dimension of $R$. They conjectured that the correct estimate should be
max$\{3, d + 2\}$, which was established by L.N. Vaserstein in \cite{V}.
Vaserstein went on to show that a similar phenomenon occurs for the
symplectic and orthogonal $K_1$ functors; and established estimates
$\geq 2d + 4$ for them.

A.A. Suslin started improving the estimates of J-P Serre and Hyman
Bass for $K_0$, by working over affine algebras over a field of cohomological
dimension at most one (under certain divisibility conditions on the field;
which we shall blur over in the introduction). This led Ravi A. Rao and
W. van der Kallen in \cite{iis}
to question whether a similar phenomenon would occur for injective
$K_1$-stabilization; which they established as max$\{3, d +1\}$. Later in \cite{sfmaa} it was further improved to max$\{3, d\}$ under certain divisibility conditions on the field. In this article we  prove similar estimate for relative $K_1$ (see Theorem \ref{13}).

Their methods were successfully
used in the symplectic group in the work of R. Basu--R.A. Rao in \cite{br}
to get the bound $\geq d + 1$ for the symplectic group. Along with S. Jose
in \cite{rbj} they also showed that the injective stability bound does not
improve for the orthogonal $K_1$ even over nice affine algebras.

The Basu--Rao results were reproved (and slightly improved) by
Chattopadhyay--Rao in \cite{cr}. Later Basu--Chattopadhyay--Rao in
\cite{bcr} succeeded in getting the injective estimate max$\{4, d\}$, for smooth
affine algebras over a field of cohomological dimension $\leq 1$, when
$d$ is divisible by $4$.

There was a further improvement in the stability estimate for $K_0$ in
the recent work of Jean Fasel--Ravi A. Rao--Richard G. Swan in \cite{sfmaa},
over a smooth affine algebra over an algebraically closed field $k$. They
also showed that the injective stability estimate fell to max$\{3, d\}$.

This led us to further investigate whether the injective stability bound
for the symplectic group also fell similarly. In this article we  first show the following result of \cite{cr} extending it to the case when $ n = 2$ and removing the condition $2R = R$.

\begin{theorem}\label{eo11}
Let $R$ be a commutative ring, $I$ an ideal of $R$ and $\v \in \Um_{2n}(R, I)$. Then  $ \v\E_{2n}(R, I) = \v\ESp_{2n}(R, I)$.
\end{theorem}

The above result is used to show the following.

\begin{theorem}\label{i52}
Let $A$ be a non singular affine algebra of dimension $d$ over an algebraically closed field $k$ such that $d \equiv 2 \vpmod 4$, $\frac{1}{d!} \in k$ and $I$ an ideal of $A$. % Let $I$ be an ideal of $A$ such that either $I$ is principal or $A/I$ is smooth. 
Let $\v \in \Um_{d}(A, I)$. Then $\v$ can be completed to a symplectic matrix which is congruent to identity modulo $I$.
\end{theorem}

Then using Theorem \ref{i52} we  show that the injective stability for symplectic $K_1$ falls to max$\{4,
d - 1\}$, when $d \equiv 1 \vpmod 4$ and $(d +1)!A = A$. Precisely we  prove the following.

\begin{theorem}\label{i53}
Let $A$ be a nonsingular affine algebra of dimension $d$ over an algebraically closed field $k$ and $I = (a)$ a principal ideal. 
Assume $(d+1)! \in k^* = k \setminus \{0\}$, $d \equiv 1 \vpmod 4$.% and $d \geq 5$. 
Then
$$\Sp_{d-1}(A, I) \cap \ESp_{d+1}(A, I) = \ESp_{d-1}(A, I).$$
\end{theorem}

We also show by an examples that one cannot expect the injective stability
bounds to fall further for smooth affine algebras over an algebraically
closed field. Just as the positive solutions of the cancellation problem
for $A^{d-1}$ leads to positive solution; the negative solution due to
Mohan Kumar in \cite{sfmk} leads to the examples that the injective stability
estimates will not fall further. These examples are also used by us to
show that the injective stability estimates do not even fall to $d -1$, as
above, when we work with affine algebras over a field of cohomological
dimension two. It remains to see if the situation improves over a field of
cohomological dimension at most one.

{\bf Throughout this article $R$ will denote a commutative ring with $1 \not= 0$ and $I$ an ideal of $R$ unless otherwise specified}.

%%%%%%%%%%%%%%%%%%%%%%%%%%%%%%%%%%%%%%%%%%%%%%%%%%%%%

\section{Preliminaries}

 A row $\v = (v_1, v_2, \ldots, v_n) \in R^n$ is said to be unimodular if there exists another row $\w = (w_1, w_2, \ldots, w_n) \in R^n$ such that $\langle \v, \w \rangle = \sum _{i =1}^{n}v_iw_i = 1$.  We shall denote the set of all unimodular rows in  $R^n$ by $\Um_n(R)$.  
Let $e_i$ denote the i - th row of the identity matrix $I_n$ of size $n$. Then
$\Um_n(R, I)$ will denote the set of all unimodular rows of length $n$ which are congruent to $e_1$ modulo $I$. It can be shown that for any $\v \in \Um_n(R, I)$ there exists  $\w \in \Um_n(R, I)$ such that $\langle \v, \w \rangle = 1$. Any subgroup $G$ of $\GL_n(R, I) = \{ \alpha \in \GL_n(R) : \alpha \equiv I_n \vpmod I \}$ acts on $\Um_n(R, I)$ where $I_n$ denotes the identity matrix.  Let $\v, \w \in \Um_n(R, I)$, we write $\v \sim_G \w$ if $\v = \w g$ for some $g \in G$.

\begin{definition}{\rm(Elementary group $\E_n(R)$, Relative elementary group $\E_n(R, I)$)}\label{elementary}
Given $\lambda \in R$, for $i \neq j$, let $E_{ij}(\lambda) = I_n + \lambda{e_{ij}}$ where $e_{ij} \in \M_n(R)$ is the matrix  whose only non-zero entry is $1$ at the $(i,j)$-th position.  Such $E_{ij}(\lambda)$'s are called elementary matrices.  The subgroup of $\GL_n(R)$ generated by  $E_{ij}(\lambda), i \not= j, \lambda \in R$ is called the elementary subgroup of $\GL_n(R)$ and is denoted  by $\E_n(R)$. Similarly we define $\E_n(I)$ for any ideal $I$ of $R$. We define $\E_n(R, I)$  to be the normal closure of $\E_n(I)$ in $\E_n(R)$. It is the smallest  normal subgroup of $\E_n(R)$ containing the element $E_{21}(x), x \in I$.
\end{definition}

\begin{definition}\label{symp}{\rm (Symplectic group $\Sp_{2n}(R)$, Relative symplectic group $\Sp_{2n}(R,I))$}
We define the symplectic group $\Sp_{2n}(R)$ to be the isotropy group of the standard symplectic form  $\psi_n = \sum^n_{i=1}{e_{2i-1\,2i}} - \sum^n_{i=1}{e_{2i\,2i-1}}$. In other words $\Sp_{2n}(R)=\{\alpha \in \GL_{2n}(R) \alpha^t\psi_n\alpha = \psi_n\}$. We define the relative symplectic group as $\Sp_{2n}(R,I)=\{\alpha \in \Sp_{2n}(R) \,| \, \alpha \equiv I_{2n} \vpmod I\}.$
Note that if $\alpha \in \Sp_{2n}(R,I)$, then its transpose $\alpha^t \in \Sp_{2n}(R,I)$. % since $(\psi_n^t)^{-1} = \psi_n$. 
\end{definition}
Let $\sigma \in S_{2n}$ denote the permutation of the natural numbers given by $\sigma(2i) = 2i-1$ and $\sigma(2i-1)=2i;  i = 1, 2, \ldots, n$.

\begin{definition}{\rm $(\ESp_{2n}(R)$, $\ESp_{2n}(R, I) )$}\label{elesymp}
We define for $z\in R$, $1\leq i \not =  j\leq2n$,
\[ 
se_{i\,j}(z) = \begin{cases}
1_{2n} + ze_{ij} & \text{if $i = \sigma(j)$};\\

1_{2n} + ze_{ij} - (-1)^{i+j}ze_{\sigma(j)\sigma(i)} & \text{if $i\not=\sigma(j)$ and $i<j$.}
\end{cases}
\]
It is easy to see that all these generators belong to $\Sp_{2n}(R)$. We call them elementary symplectic matrices over $R$, and the subgroup of $\Sp_{2n}(R)$ generated by them is called the elementary symplectic group $\ESp_{2n}(R)$. Similarly the subgroup generated by $se_{ij}(z), z \in I$ is denoted by $\ESp_{2n}(I)$. The group $\ESp_{2n}(R, I)$ is defined to be the smallest normal subgroup of $\ESp_{2n}(R)$ containing $\ESp_{2n}(I)$.
\end{definition}

The following is easy.
\begin{lemma}\label{rcm}
If the first row of a symplectic matrix $M\in \Sp_{2n}(R, I)$ equals $e_1$, then its second column is equal to $e_2^t$ and vice versa.
\end{lemma}

%\begin{proof}
% Let the first row of $M$ be $e_1$. We have $M^t\psi_n M = \psi_n = M\psi_n M^t $, $e_1M = e_1$ and $e_1 \psi_n = e_2$. So $e_1M\psi_n = e_2$. Therefore $e_2 M^t=e_1M\psi_n M^t = e_1 \psi_n = e_2$. This means $Me_2^t= e_2^t$. In other words second column of $M$ equals $e_2^t$. The converse follows similarly.
%\end{proof}

\begin{definition}\label{7a}{\rm(The excision ring)}
If $I$ is an ideal of $R$, one constructs the ring $\ZZ \oplus I$ with multiplication defined by 
$$(n,  i)(m,  j) = (nm,  nj+mi+ij)$$ for $m,n \in \ZZ, i,j \in I$.
If the dimension of the ring is $d \geq 1$, then the maximal spectrum of $\ZZ \oplus I$ is the union of finitely many subspaces of dimension at most $d$ {\rm (\cite{gsu}, 3.19)}.  

There is a natural homomorphism $\mathfrak{f}: \ZZ \oplus I \longrightarrow R$ given by $(m , i) \longmapsto m+i \in R$. Let $\v = (1 +i_1, i_2, \ldots, i_n) \in \Um_n(R, I)$ where $i_j$'s are in $I$. Then we shall say $\tilde{\v} = (\tilde{1} + \tilde{i_1}, \tilde{i_2}, \ldots, \tilde{i_n}) \in \Um_n(\ZZ \oplus I, 0 \oplus I)$ for $\tilde{1} = (1, 0), \tilde{i_j} = (0, i_j)$ to be a lift of $\v$. Clearly $\mathfrak{f}$ sends $\tilde{\v}$ to $\v$.
\end{definition}

We shall define $\MSE_n(R) = \frac{\Um_n(R)}{\E_n(R)}$ and likewise $\MSE_n(R, I) = \frac{\Um_n(R, I)}{\E_n(R, I)}$. We recall the Excision theorem (see \cite{gsu}, Theorem
$3.21$). 

\begin{theorem}\label{7b}{\rm (Excision theorem)} Let $n \geq 3$
be an integer and $I$ be an ideal of a commutative ring $R$. Then the
natural maps  $F : \MSE_n(\ZZ \oplus I, 0 \oplus I) \rightarrow \MSE_n(R,
I)$ {\rm defined by} $[(a_i)] \mapsto [(\mathfrak{f}(a_i))]$ {\rm and}  $G : \MSE_n(\ZZ
\oplus I, 0 \oplus I) \rightarrow \MSE_n(\ZZ \oplus I)$ {\rm defined by} $ [(a_i)] \mapsto [(a_i)]$ are
bijections. 
	
\end{theorem} 

\begin{definition}\label{retract} We shall say a ring homomorphism $\phi: B
\twoheadrightarrow D$ has a section if there exists a ring homomorphism
$\gamma : D \hookrightarrow B$ so that $\phi \circ \gamma $ is the
identity on $D$. We shall also say that $D$ is a retract of $B$.
\end{definition} 

The following is easy.

\begin{lemma}\label{lem:august15}  Let $B, D$ be rings and
and  $\pi : B \twoheadrightarrow D$ has a section.  If $J = {\rm
ker}(\pi)$, then $\E_{n}(B, J) = \E_{n}(B) \cap \SL_{n}(B, J), n \geq 3$ and $\ESp_{2n}(B, J) = \ESp_{2n}(B) \cap \Sp_{2n}(B, J), n \geq 2$.
\end{lemma} 

It is well known that the double of a ring w.r.t. 
an ideal is the same as the excision algebra w.r.t. an ideal $I$. The following is proved in (\cite{cpm}, Proposition 3.1). 

\begin{proposition}\label{9a} 
Let $R$ be a ring of
dimension $d$ and $I$ a finitely generated ideal of $R$. 

Consider the Cartesian square 
$$\begin{CD}
C @>>> R \\ @VVV  @VVV\\ R @>>> R/I 
\end{CD}$$  Then, $C$ is a finitely generated algebra of dimension $d$
over $R$ and integral over $R$. In fact, $C \simeq R \oplus I$ with the
coordinate wise addition and the multiplication defined by $(a, i)(b,j) = (ab, aj + ib +ij)$, $\tilde{1} = (1, 0)$ being the identity in $C$. In particular, if $R$ is an affine algebra of dimension $d$ over a field $k$, then $C \simeq R \oplus I$ is also an affine algebra of dimension $d$ over $k$.
\end{proposition}

\begin{remark}\label{9a1a}
We shall call $C$ as the excision algebra of $R$ with respect to the ideal $I$. There is a natural homomorphism $\mathfrak{g}: R \oplus I \longrightarrow R$ given by $(x, i) \longmapsto x+i \in R$. Clearly $\mathfrak{g}$ has a section.
%For notational simplicity we shall also denote the induced maps $(R \oplus I)[X] \longrightarrow R[X]$, $\M_n(R \oplus I) \longrightarrow \M_n(R)$ by $\mathfrak{g}$.

Let $\v = (1 +i_1, i_2, \ldots, i_n) \in \Um_n(R, I)$ where $i_j$'s are in $I$. Then we shall call $\tilde{\v} = (\tilde{1} + \tilde{i_1}, \tilde{i_2}, \ldots, \tilde{i_n}) \in \Um_n(R \oplus I, 0 \oplus I)$ for $\tilde{1} = (1, 0), \tilde{i_j} = (0, i_j)$ to be a lift of $\v$. Note that $\mathfrak{g}$ sends $\tilde{\v}$ to $\v$.
\end{remark}

The following is from (\cite{sppmp}, Section \S 5,  Remarks (a)).

\begin{lemma}{\rm(Vaserstein)}\label{eo10}
For any commutative ring $R$, we have 
$$e_1\E_{2n}(R) = e_1\ESp_{2n}(R).$$
\end{lemma}

\begin{lemma}\label{eo10a}
For any commutative ring $R$ and an ideal $I$ of $R$, we have $$e_1\E_{2n}(R, I) = e_1\ESp_{2n}(R, I).$$
\end{lemma}

\begin{proof}
Let $R^* = R \oplus I$ and $I^* = 0 \oplus I$. Clearly the quotient  map $q: R^* \longrightarrow R = R^*/ I^*$ has a section.
Let $\v \in e_1\E_{2n}(R, I)$. Then there exists  $\varepsilon \in \E_{2n}(R, I)$ such that $\v = e_1 \varepsilon$. Let $\tilde{\varepsilon} \in \E_{2n}(R^*, I^*)$ be a lift of $\varepsilon$ i.e. $\mathfrak{g}(\tilde{\varepsilon}) =\varepsilon$. 
%(see Remark \ref{9a1a}).
Then $\tilde{\v} = e_1\tilde{\varepsilon} \in \Um_{2n}(R^*, I^*)$ is a lift of $\v$. Now by Lemma \ref{eo10} we have  $\tilde{\delta} \in \ESp_{2n}(R^*)$ such that $\tilde{\v} = e_1 \tilde{\delta}$. Going modulo $I^*$ we have $e_1 = e_1 
q(\tilde{\delta})$. So replacing $\tilde{\delta}$ by $q(\tilde{\delta})^{-1} \tilde{\delta}$ if necessary we may assume that $\tilde{\delta} \in \ESp_{2n}(R^*) \cap \Sp_{2n}(R^*, I^*) = \ESp_{2n}(R^*, I^*)$ by Lemma \ref{lem:august15}. So $\delta = \mathfrak {g}(\tilde{\delta}) \in \ESp_{2n}(R, I)$ and $\v = \mathfrak{g}(\tilde{\v}) = e_1 \mathfrak{g}(\tilde{\delta}) = e_1 \delta \in e_1\ESp_{2n}(R, I)$. Therefore $e_1\E_{2n}(R, I) \subset e_1\ESp_{2n}(R, I)$. The reverse inclusion is obvious.
\end{proof}

The following lemma follows from (\cite{V}, Theorem 2.3(e), 2.5), (\cite{gsu}, Proposition 2.4).

\begin{lemma}\label{1}
Let the maximal spectrum $max(R)$ of $R$ be a disjoint union of $V(I) = \{\mathfrak{m} \in max (R) : I \subset \mathfrak{m}\}$ and finitely many  subsets $V_i$ where each $V_i$, when endowed with the  Zariski topology $($topology induced from Spec $R)$ is a noetherian space of dimension   at most $d$. Let $(a_0, a_1, \ldots, a_m) \in \Um_{m+1}(R)$, $a_0 \equiv 1 \vpmod I$ and $m \geq d +1$ and $S$ a subset of size $d+1$ of $\{0, 1, \ldots, m - 1\}$. Then there exists $t_i \in I$ with $t_i = 0$ for $i \not \in S$ such that  $(a_0 + a_mt_0, a_1 + a_m t_1, \ldots, a_{m-1} + a_mt_{m-1}) \in \Um_m(R)$. In particular $\Um_m(R,I) = e_1\E_m(R,I)$ for $m \geq d+2$.
\end{lemma}

\begin{lemma}\label{5}
Let the maximal spectrum of $R$ satisfy conditions as in Lemma \ref{1} and $\alpha \in \Sp_{2n}(R, I)$ for $2n \geq d+2$. Then $\exists$  $\varepsilon \in \ESp_{2n}(R, I)$ such that $\alpha \varepsilon = I_2 \perp \gamma$ for some $\gamma \in \Sp_{2n-2}(R, I)$.
\end{lemma}

\begin{proof}
Let $\v = e_1 \alpha \in \Um_{2n}(R,I)$. Since $2n \geq d+2$ we have $\Um_{2n}(R,I) = e_1\E_{2n}(R, I) = e_1\ESp_{2n}(R, I)$ by Lemmas \ref{eo10a}, \ref{1}. So we have an elementary symplectic matrix $\varepsilon' \in \ESp_{2n}(R,I)$ such that $\v\varepsilon' = e_1$. Therefore the first row of the symplectic matrix $\alpha \varepsilon'$ is $e_1$ and by Lemma \ref{rcm} its second column is $e_2^t$. So $\alpha\varepsilon'$ will look like 
$\begin{pmatrix} 1 & 0 & 0\\
                 x & 1 &v_1 \\
                 v_2^t &0 &\beta\\
\end{pmatrix}$ for some $x \in I$, $v_1, v_2 \in \M_{1 \;n-2}(I)$ and $\beta \in \GL_{n-2}(R, I)$.
Now we shall multiply $\alpha \varepsilon'$ by suitable elementary symplectic matrices of the form $se_{2j}(z), j \not= 2, z \in I$ from the right to change all non diagonal entries of the second row to zero. Note that no entries of the first row gets affected for these multiplications. The resulting matrix is a symplectic matrix whose second row is $e_2$ and therefore by Lemma \ref{rcm} its first column must be $e_1^t$. Thus we have an elementary symplectic matrix $\varepsilon$ such that $\alpha \varepsilon = I_2 \perp \gamma$ for some $\gamma \in \Sp_{2n-2}(R, I)$.
\end{proof}

\begin{corollary}\label{6}
If  $2n \geq d+2$, then we have $\varepsilon \in \ESp_{2n}(R,I)$ such that
\begin{eqnarray*}
\alpha\varepsilon = 
\begin{cases}
I_{2n-d}\perp\gamma\text{ for some} \,\gamma \in \Sp_d(R,I)\, & \text{when  $d$ is even}\\
I_{2n-d-1} \perp \gamma \text{ for some} \,\gamma \in \Sp_{d+1}(R,I)\, &  \text{when $d$ is odd}.
\end{cases}
\end{eqnarray*}
\end{corollary}

\begin{definition}{\rm(Permutation matrices)}
It is well known that there exists an injective group homomorphism from the permutation group $S_n$ on $n$ symbols to $\GL_n(\Z)$ defined by $$\sigma \longmapsto (e_{\sigma(1)}^t, e_{\sigma(2)}^t,\ldots, e_{\sigma(n)}^t) = (e_{\sigma^{-1}(1)}, e_{\sigma^{-1}(2)},\ldots, e_{\sigma^{-1}(n)})^t.$$ %Here $e_i$ denotes the $i$-th row of the identity matrix $I_n$. 
If we identify each permutation with its image in $\GL_n(\Z)$ then we can view $S_n$ as a subgroup of $\GL_n(\Z)$. Matrices in $S_n$  are called permutation matrices.

A transposition $(i\, j)$ corresponds to $E_{ji}(1) E_{ij}(-1) E_{ji}(1)\delta_j$ where $\delta_j$ is a diagonal matrix with all but $(j,j)$-th entry 1 and $(j,j)$-th entry $-1$. If $\sigma$ is an even permutation matrix then $\sigma \in \SL_n(\Z) = \E_n(\ZZ)$ as each transposition is of determinant $-1$.
% Also $\sigma = \varepsilon \delta, \varepsilon \in \E_n(\Z)$ and $\delta$ is a diagonal matrix with entries $\pm 1$ and of determinant $1$. Now it is well known that any diagonal matrix of determinant $1$ is an elementary matrix $($\cite{iar}$)$. 
 Thus the group of even permutation matrices is a subgroup of $\E_n(\Z)$.
\end{definition}

\begin{definition}\label{6a}{\rm(Suslin Matrix)} Given two vectors $\v, \w \in R^{r+1}$  A.A.Suslin in (\cite{sfm}, \S 5) gave an inductive process to construct the Suslin matrix $S_r(\v, \w)$. We recall this process: Let $\v=(a_0,\v_1)$, $\w=(b_0, \w_1)$, where $a_0, b_0 \in R$ and $\v_1, \w_1 \in \M_{1r}(R)$. Set $S_r(\v, \w) =(a_0)$ for $r =0$ and define
\[S_r(\v, \w)=
\begin{pmatrix}
a_0I_{2^{r-1}}&S_{r-1}(\v_1, \w_1)\\
-S_{r-1}(\w_1,  \v_1)^t &b_0I_{2^{r-1}}
\end{pmatrix}
\]
\end{definition}

In (\cite{sfm}, Lemma 5.1) it is noted that $S_r(\v, \w)S_r(\w, \v)^t=\langle \v, \w \rangle I_{2^r}$= $S_r(\w, \v)^t S_r(\v, \w)$  and $\text{det}\, S_r(\v, \w)$ $=\langle \v,  \w \rangle^{2^{r-1}}$ for $r\geq1$. A.A. Suslin introduced these matrices and showed that a unimodular row of the form $(a_0, a_1, a_2^2,\ldots,a_r^r)$ can be completed to an invertible matrix. Following Vaserstein in \cite{V4} we shall call unimodular row of this type as factorial row and denote it by $\psi_{r!}(\v)$ for $\v=(a_0, a_1, \ldots, a_r)$. In fact in (\cite{ms}, Proposition 2.2, Corollary 2.5) it is shown that there is an $\beta_r(\v, \w)$, $\langle \v, \w \rangle = 1$ with $[\beta_r(\v, \w)] = [S_r(\v, \w)]$ in $K_1(R)$ whose first row is $\psi_{r!}(\v)$. In (\cite{seh}, Corollary 4.3(iii)) it is shown that if $\v, \w \in \Um_{r+1}(R,I), r \geq 2$ then we may assume that $\beta_r(\v, \w) \equiv I_{r+1}$ modulo $I$.

A.A Suslin then described a sequence of forms $J_r \in M_{2^r}(R)$ by the recurrence formula:
\[
J_r=
\begin{cases}
1 & \text{for $r=0$},\\
J_{r-1} \perp - J_{r-1} & \text{for $r$ even}\\
J_{r-1} \top - J_{r-1} & \text{for $r$ odd}
\end{cases}
\]
(Here $\alpha \perp \beta = \begin{pmatrix} \alpha & 0\\0 & \beta \end{pmatrix}$ while $\alpha \top \beta = \begin{pmatrix} 0 & \alpha \\ \beta &0 \end{pmatrix}$.)

It is easy to see that $\det(J_r)=1$ for all $r$, and that $J_r^t = J_r^{-1}=(-1)^{\frac{r(r+1)}{2}}J_r$. Moreover $J_r$ is antisymmetric if $r= 4k+1$ and $r=4k+2$ whereas $J_r$ is symmetric if $r=4k$ and $r=4k+3$. In (\cite{sfm}, Lemma 5.3) it is noted that the following formulas are valid: The Suslin identities are 

\begin{align*}
&\text{for}\ r=4k :   (S_r(\v, \w)J_r)^t = S_r(\v, \w)J_r;\\
&\text{for}\ r=4k+1 :S_r(\v, \w)J_rS_r(\v, \w)^t =  \langle \v, \w\rangle J_r;\\
&\text{for}\ r=4k+2: (S_r(\v, \w)J_r)^t = - S_r(\v, \w)J_r;\\
&\text{for} \ r=4k+3: S_r(\v, \w)J_rS_r(\v, \w)^t =  \langle \v, \w \rangle J_r.
\end{align*}

\begin{remark}\label{7}
Consider the Suslin matrix $S_r(\v, \w)$, with $\langle \v, \w \rangle =1$. By the Suslin identities $S_r(\v, \w)$ is a symplectic matrix w.r.t $J_r$ when $r \equiv 1 \vpmod 4$  and an orthogonal matrix w.r.t $J_r$ when $r \equiv 3 \vpmod 4$.
It is to be noted that $J_r$ is merely a permutation matrix if we ignore the sign of its entries. So we shall have a permutation matrix $\sigma_{J_r}$ such that $\sigma_{J_r} J_r\sigma_{J_r}^t = \psi_{2^{r-1}}$ when $J_r$ is antisymmetric i.e. $r=4k+1, 4k+2$. Here $\psi_n = \sum^n_{i=1}{e_{2i-1\,2i}} - \sum^n_{i=1}{e_{2i\,2i-1}}$.

When $r=4k + 1,\ 4k+2$, we shall denote the group of matrices $\alpha \in \GL_{2n}(R)$ satisfying $\alpha J_r \alpha^t = J_r$ by $\Sp_{J_r}(R)$. Clearly 
$\Sp_{J_r}(R) = \{\alpha | \alpha J_r \alpha^t = J_r\} = \sigma_{J_r}^{-1}\Sp_{2^r}(R) \sigma_{J_r}$. We shall define $\ESp_{J_r}(R) = \sigma_{J_r}^{-1}\ESp_{2^r}(R) \sigma_{J_r}$, $\ESp_{J_r}(R, I) = \sigma_{J_r}^{-1}\ESp_{2^r}(R, I) \sigma_{J_r}$. By $($\cite{sslp}, Corollary 1.4$)$ and Theorem \ref{7b}  we have $\ESp_{J_r}(R, I) \subset \E_{2^r}(R, I)$.
\end{remark}

The following is proved in (\cite{seh}, Lemma 4.2).

\begin{lemma}\label{8}
Let $R$ be a ring and $J$ an ideal of $R$. Let $\alpha \in \SL_n(R)$ with $\alpha \equiv I_n \vpmod J$. Then there exists a unique matrix $S \in \SL_n(\ZZ \oplus J)$ such that $S \equiv I_n \vpmod {0 \oplus J}$, and with $\mathfrak{f}(S) = \alpha$, where $\mathfrak{f}:\ZZ \oplus J \longrightarrow R$ is the natural homomorphism (see Definition \ref{7a}).
 Moreover, if $n=2^m$, for some $m$, and $\alpha$ is a Suslin matrix of determinant one, then $S$ is also a Suslin matrix of determinant one.
\end{lemma}

%\begin{lemma}\label{7}
%Let the maximal spectrum of $R$ satisfy conditions as in Lemma \ref{1}. Assume that $S_r(v,w) \in \SL_{2^r}(R,I), 2^r\geq d+2$ for $v,w \in \Um_{r+1}(R,I)$ such that $\langle v, w \rangle=1$. Further assume that $r \equiv 1 \vpmod 4$. Then there exists $\varepsilon_{J_r} \in \ESp_{J_r}(R, I)$ such that $S_r(v,w)\varepsilon_{J_r} = \sigma_{J_r}^{-1} (I_{2^r - k} \perp \gamma)\sigma_{J_r}$ for some $\gamma \in \Sp_{k}(R,I)$. Here $k = d+1, d$ according as $d$ is odd or even respectively.
%\end{lemma}
%\begin{proof}
%Since $r=4k+1$ we have $S_r(v,w) \in \Sp_{J_r}(R,I)$. So $S = \sigma_{J_r} S_r(v,w) \sigma_{J_r}^{-1} \in \Sp_{2^r}(R,I)$. Therefore by Corollary \ref{6} we have an $\varepsilon \in \ESp_{2^r}(R,I)$ such that $S \varepsilon = I_k \perp \gamma$ for some symplectic matrix $\gamma$ of desired size. If we define $\varepsilon_{J_r} = \sigma_{J_r}^{-1} \varepsilon \sigma_{J_r} \in \ESp_{J_r}(R, I)$ then we have $S_r(v,w)\varepsilon_{J_r} = \sigma_{J_r}^{-1} (I_k \perp \gamma)\sigma_{J_r}$.
%\end{proof}

We recall (\cite{bcr}, Lemma 2.13) and include a proof for the sake of completeness.

\begin{lemma}\label{9}
Let the maximal spectrum of $R$ satisfy conditions as in Lemma \ref{1}, $d \geq 1$ and $S_r(\v, \w) \in \SL_{2^r}(R,I)$, $2^r\geq d+2$ for $\v, \w \in \Um_{r+1}(R,I)$ satisfying $ \langle \v, \w \rangle =1$. Assume that $r\equiv 1  \vpmod 4$. Then there exists $\varepsilon_{J_r} \in \ESp_{J_r}(R, I)$ such that $S_r(\v, \w)\varepsilon_{J_r} =(I_{2^r - k} \perp \gamma)\varepsilon$ for some $\gamma \in \Sp_k(R,I)$ and $\varepsilon \in \E_{2^r}(R,I)$. Here $k = d+1, d$ according as $d$ is odd or even respectively.
\end{lemma}

\begin{proof}
Let $R^* = \ZZ \oplus I$, $I^* = 0 \oplus I$. We shall consider the corresponding Suslin matrix $S_r(\v^*, \w^*) \in \SL_{2^r}(R^*, I^*)$ described in the Lemma \ref{8}. Since $r=4k+1$ we have $S_r(\v^*, \w^*) \in \Sp_{J_r}(R^*, I^*)$.
So $S = \sigma_{J_r}S_r(\v^*, \w^*) \sigma_{J_r}^{-1} \in \Sp_{2^r}(R^*, I^*)$. Therefore by Corollary \ref{6} we have an $\varepsilon \in \ESp_{2^r}(R^*, I^*)$ such that $S \varepsilon = I_{2^r - k} \perp \gamma^*$ for some symplectic matrix $\gamma^* \in \Sp_k(R^*,I^*)$ of desired size. If we define $\varepsilon_{J_r} = \sigma_{J_r}^{-1} \varepsilon \sigma_{J_r} \in \ESp_{J_r}(R^*, I^*)$ then we have $S_r(R^*, I^*)\varepsilon_{J_r} = \sigma_{J_r}^{-1} (I_{2^r - k} \perp \gamma)\sigma_{J_r}$.

%By Lemma \ref{7} we have $S_r(v^*,w^*)\varepsilon^*_{J_r} = \sigma_{J_r}^{-1} (I_{2^r - k} \perp \gamma^*)\sigma_{J_r}$ for some $\varepsilon^*_{J_r} \in \ESp_{J_r}(R^*, I^*)$ and $\gamma^* \in \Sp_k(R^*,I^*)$. 
Now after interchanging the first pair of rows and columns if necessary we may assume that $S_r(\v^*,\w^*)\varepsilon^*_{J_r} = \sigma^{-1} (I_{2^r - k} \perp \gamma^*)\sigma$ for some even permutation $\sigma \in \E_{2^r}(\Z)$. Let $\varepsilon^* = \{(I_{2^r - k} \perp \gamma^*)^{-1}\sigma^{-1}(I_{2^r - k} \perp \gamma^*)\}\sigma$. We know that $\E_{2^r}(R^*)$ is a normal subgroup of $\SL_{2^r}(R^*)$  (\cite{sslp}, Corollary 1.4). So $\varepsilon^* \in \E_{2^r}(R^*) \cap \SL_{2^r}(R^*, I^*)$. But $R^*/I^*$ is a retract of $R^*$ (Definition \ref{retract}). Therefore $\varepsilon^* \in \E_{2^r}(R^*, I^*)$ by Lemma \ref{lem:august15} and we have $S_r(\v^*,\w^*)\varepsilon^*_{J_r} =(I_{2^r - k} \perp \gamma^*)\varepsilon^*$. Taking the image in $R$ under the natural homomorphism $\mathfrak{f} : R^* \longrightarrow R$ the result follows.
\end{proof}

%%%%%%%%%%%%%%%%%%%%%%%%%%%%%%%%%%%%%%%%%%%%%%%%%%%%%%%%%%%%%%%%%%%%%%%
\section{Equality of orbits}

In the previous section we have seen that both the groups $\E_{2n}(R, I)$ and $\ESp_{2n}(R, I)$ act on $\Um_n(R, I)$. It is natural to ask whether orbits under these two different group actions are the same. In (\cite{cr}, Theorem 4.1), it is shown that this is indeed the case in the absolute case i.e. $\v\E_{2n}(R) = \v\ESp_{2n}(R)$ for $\v \in \Um_{2n}(R), n \geq 2$. In the same paper it is also proved that the result is true in the relative case too but with an extra hypothesis viz $n \geq 3$ and $2R = R$ (see Theorem 5.5). In this section we give a different proof of this result which works both in the absolute case and relative case and also in all characteristics.

\begin{definition}\label{eo1}
Let $\v = (v_0, v_1, \ldots, v_{2n-1}) \in R^{2n}$. Then we define $\widehat{\v} = \v \psi_{2n}$. Note that $\langle \v, \widehat{\v} \rangle = 0$. % \langle \widehat{v}, v \rangle = 0$.
\end{definition}

\begin{lemma}{\rm(\cite{sslp}, Lemma 1.3)}\label{eo2}
Let $\v = (v_1, v_2, \ldots,v_n) \in \Um_n(R)$ and $\u = (u_1, u_2,$ $ \ldots, u_n) \in \Um_n(R)$ be such that $\sum_{i=1}^nu_iv_i = 1$. Let $\phi : R^n \rightarrow R$ be the map sending $e_i \mapsto v_i$ where $\{e_1, e_2, \ldots, e_n\}$ is the basis of $R^n$. Then for $\w = (w_1, w_2, \ldots, w_n) \in ker(\phi)$ we have $\w = \sum_{i \not= j}w_iu_j(v_je_i - v_ie_j)$.
\end{lemma}

\begin{lemma}{\rm(\cite{vik}, Lemma 1.5)}\label{eo3}
Let $n \geq 2$ and $a \in I$, $\v \in R^{2n}$ or $a \in I$, $\v \in I^{2n}$. Then $1_{2n} + a \v^t \widehat{\v} \in \ESp_{2n}(R, I)$.
\end{lemma}

\begin{lemma}{\rm(\cite{vik}, Lemma 1.9)}\label{eo4}
Let $\w \in R^{2n}$ and $\v_i \in I^{2n}, n \geq 2$ for $i = 1, 2, \ldots k$ be such that $\langle \widehat{\v_i},  \w \rangle = 0$ for all $i$. Then there is an element $b \in I$ such that \[ 1_{2n} + \sum_{i = 1}^k (\v_i ^{t} \widehat{\w} + \w ^{t} \widehat{\v_i}) = \prod_{i=1}^k ( 1_{2n} + \v_i ^{t} \widehat{\w} + \w ^{t} \widehat{\v_i}) \times (1_{2n} + b \w ^{t} \widehat{\w}) \]
\end{lemma}

%\rm{(Dilation Principle)}
\begin{lemma}\label{eo5}
 Let $\alpha(X) \in \ESp_{2n}(R[T, T^{-1}, X], (X)), n \geq 2$. Then  for sufficiently large $N$ we have $\alpha(T^NX) \in \ESp_{2n}(R[T, X], (X))$. 
\end{lemma}
\begin{proof}
Let $S = R[T, T^{-1}]$. It is enough to prove the theorem for $\alpha(X) = \gamma se_{ij}(Xf)\gamma^{-1}$, $i \not= j$, $f \in S[X], \gamma \in \ESp_{2n}(S[X])$. Note that if $\v_j = e_j\gamma^t$ i.e. $\v_j^t$ is the $j^{\text{th}}$ column of $\gamma$ then $(-1)^{j+1}\widehat{\v_j}$ is the $\sigma(j)^{\text{th}}$  row of $\gamma^{-1}$. This is because $\gamma^t\psi_{2n}\gamma = \psi_{2n}$ gives $\widehat{\v_j} = e_j \gamma ^t \psi_{2n} = e_j \psi_{2n} \gamma^{-1} = (-1)^{j+1}e_{\sigma(j)} \gamma^{-1}$. We prove it in the following cases.

\noindent
{\bf Case - 1} $i = \sigma(j)$\\
In this case $\alpha(X) = \gamma se_{i\sigma(i)}(Xf(X))\gamma^{-1} = 1_{2n} +(-1)^{i + 1} Xf(X) \v_i^t \widehat{\v_i}$. We choose $N$ sufficiently large such that $T^N\v_i \in \M_{1\,2n}(R[X, T])$ and $T^Nf(X) \in R[X, T]$. Then $\alpha(T^{3N}X) \in  \ESp_{2n}(R[T, X], (X))$ by Lemma \ref{eo3}.

\noindent
{\bf Case - 2} $i \not= \sigma(j)$\\
In this case  $\alpha(X) = \gamma se_{ij}(Xf(X))\gamma^{-1} = 1_{2n} + (-1)^{j} Xf(X)\v_i^t \widehat{\v_{\sigma(j)}} +  (-1)^{j} Xf(X)\v_{\sigma(j)}^t \widehat{\v_i}$. Call $\v_i = \v = (v_1, v_2, \ldots, v_{2n}) \in \Um_{2n}(S[X])$ and $\widehat{\v_{\sigma(j)}} = \w = (w_1, w_2, \ldots, w_{2n}) \in \Um_{2n}(S[X])$. Note that $\langle \w, \v \rangle = 0$ since $(-1)^j\w$ is the $j^{th}$ row of $\gamma^{-1}$. Now by Lemma \ref{eo2} $\w = \sum_{k \not= l}w_ku_l(v_le_k - v_ke_l)$ where $\u = (u_1, u_2, \ldots, u_{2n}) \in \Um_{2n}(S[X])$ such that $\langle \u, \v \rangle = 1$. We have $\v_{\sigma(j)} = - \w \psi_{2n} = \sum_{k \not= l}w_ku_l((-1)^k v_le_{\sigma(k)} - (-1)^l v_ke_{\sigma(l)})$.  Let $\v_{kl} =  w_ku_l((-1)^k v_le_{\sigma(k)} - (-1)^l v_ke_{\sigma(l)})$. Then $\widehat{\v_{kl}} = w_k u_l(v_le_k - v_ke_l)$ and $\langle \widehat{\v_{kl}}, \v \rangle = 0$. Note that $\v_{\sigma(j)} = \sum_{k \not= l} \v_{kl}$ and $\w = \widehat{\v_{\sigma(j)}} =  \sum_{k \not= l} \widehat{\v_{kl}}$.                 
\begin{eqnarray*}
&&\alpha(X) = 1_{2n} + (-1)^{j} Xf(X)\v_i^t \widehat{\v_{\sigma(j)}} +  (-1)^{j} Xf(X)\v_{\sigma(j)}^t \widehat{\v_i}\\
%&&= 1_{2n} + (-1)^j Xf(X) \sum_{k \not= l}\left[\v^t w_k u_l(v_le_k - v_ke_l) + w_ku_l  \{(-1)^kv_le_{\sigma(k)} - (-1)^l v_ke_{\sigma(l)}\}^t \widehat{\v} \right]\\
&&= 1_{2n} +  (-1)^j Xf(X) \sum_{k \not= l}\left[\v^t \widehat{\v_{kl}} + \v_{kl}^t \widehat{\v} \right]\\
&&= \prod_{k \not= l} \left[1_{2n} + (-1)^j Xf(X)  (\v_{kl}^t  \widehat{\v} +  \v^t  \widehat{\v_{kl}})\right]
\left(1_{2n} + Xh(X) \v^t \widehat{\v}\right)\, \text{by Lemma \ref{eo4}.}
\end{eqnarray*}
We choose $N$ sufficiently large such that $T^N\v_{kl}, T^N\v$ are in $\M_{1\,2n}(R[X, T])$ and $T^Nf(X)$, $T^Nh(X) \in R[T, X]$ . Then $\alpha(T^{3N}X) \in  \ESp_{2n}(R[T, X], (X))$ by an argument which is implicit in  (\cite{vik}, Lemma 1.10).
\end{proof}

\begin{corollary}\rm{(Dilation Principle)}\label{eo6}
 Let $\alpha(X) \in \ESp_{2n}(R_a[X], (X))$. Then  for sufficiently large $N$ we have $\beta(X) \in \ESp_{2n}(R[X], (X))$  such that $\alpha(a^NX) = \beta(X)_a$. 
\end{corollary}
\begin{proof}
We choose $\gamma(X, T) \in \ESp_{2n}(R[X, T, 1 /T], (X))$ such that $\gamma(X, a) = \alpha(X)$. Then by Lemma \ref{eo5} we have an integer $N$ sufficiently large such that   $\gamma(T^NX, T) \in \ESp_{2n}(R[X, T], (X))$. Now the result follows for $\beta(X) = \gamma(a^NX, a)$. 
\end{proof}

Using the dilation principle one can prove the following. The proof is similar to Vaserstein's proof of the corresponding result for  elementary action (see \cite{lam}, Chapter III, Section \S 2, Proposition 2.3, Theorem 2.4, 2.5).

\begin{proposition}\label{eo7} 
Let $\v(X) \in \Um_{2n}(R[X]), n \geq 2$ such that $\v(X)_s \sim_{\ESp_{2n}(R_s[X])} \v(0)_s$. Then for sufficiently large $m$ we have $\v(X + s^mY) \sim_{\ESp_{2n}(R[X, Y])} \v(X)$. 

Let $S$ be a sub ring of $R$. Define $\mathfrak{A} = \{a \in S \, | \, v(X)_a \sim_{\ESp_{2n}(R_a[X])} v(0)\}$ and $\mathfrak{B} = \{b \in S \ | \ \v(X + bY) \sim_{\ESp_{2n}(R[X, Y])} \v(X) \}$. 
Then $\mathfrak{A}, \mathfrak{B}$ are ideals of $S$ such that $\mathfrak{A} = \sqrt{\mathfrak{B}}$. 

Suppose $\v(X)_{\mathfrak{m}} \sim_{\ESp_{2n}(R_{\mathfrak{m}}[X])} \v(0)_{\mathfrak{m}}$ for all maximal ideals $\mathfrak{m}$ of $S$. Then $\mathfrak{A} = \mathfrak{B} = S$. In particular $\v(X + Y) \sim_{\ESp_{2n}(R[X, Y])} \v(X)$ and therefore $\v(X) \sim_{\ESp_{2n}(R[X])} \v(0)$.
\end{proposition}

The following  generalizes Proposition \ref{eo7}.

\begin{proposition}\label{eo8}
Let $R = \oplus_{ i \geq 0} R_i$ be a positively graded ring and $R_{+} = \oplus_{ i \geq 1} R_i$ the augmentation ideal of $R$. Let  $q : R \longrightarrow R_0 = R/ R_{+}$ be the quotient map, $\v \in \Um_{2n}(R)$ and $\w =q(\v) \in \Um_{2n}(R_0)$. Suppose $\v_{\mathfrak{m}} \sim_{\ESp_{2n}(R_{\mathfrak{m}})} \w_{\mathfrak{m}}$ for all maximal ideals $\mathfrak{m}$ of $R_0$. Then $\v \sim_{\ESp_{2n}(R)} \w$.
\end{proposition}

\begin{proof}
Let $\phi : R[T] \twoheadrightarrow R$ be the evaluation map sending $T$ to $1$. Then $\phi$ has a section $\psi : R \hookrightarrow R[T]$ defined by $\psi(\sum_{i = 0}^r x_i) = \sum_{i = 0}^r x_iT^i$, deg $x_i = i$ i.e. $\phi \psi = id$.

Let $\u(T) = \psi(\v) \in \Um_{2n}(R[T])$. Then by our hypothesis $\u(T)_{\mathfrak{m}} \sim_{\ESp_{2n}(R_{\mathfrak{m}}[T])} \u(0)_{\mathfrak{m}}$ for all maximal ideals $\mathfrak{m}$ of $R_0$. Therefore by Proposition \ref{eo7} we have $\u(T) \sim_{\ESp_{2n}(R[T])} \u(0)$. The result now follows taking the image of both sides under $\phi$.
\end{proof}

%The following is  a special case of (\cite{sppmp}, Lemma 5.4).
%
%\begin{lemma}\label{eo9}
%Let $\v \in R^{2n - 1}$. Then there exists $\alpha, \beta \in \SL_{2n - 1}(R)$ such that 
%\[\begin{pmatrix} 1 & 0 \\ \alpha \v^t & \alpha \end{pmatrix}, \begin{pmatrix} 1 & \v \\ 0 & \beta \end{pmatrix} \in \ESp_{2n}(R).\]
%\end{lemma}
%\begin{proof}
%Note that $\left(\begin{smallmatrix} 1 & 0 \\ \alpha \v^t & \alpha \end{smallmatrix}\right) = \left(\begin{smallmatrix}1 & 0 \\  0 & \alpha \end{smallmatrix}\right) \left(\begin{smallmatrix} 1 & 0 \\  \v^t & I_{2n - 1} \end{smallmatrix}\right)$ and  $\left(\begin{smallmatrix} 1 & \v \\ 0 & \beta  \end{smallmatrix}\right) = \left(\begin{smallmatrix}1 & 0 \\  0 & \beta \end{smallmatrix}\right) \left(\begin{smallmatrix} 1 & \v \\ 0 & I_{2n - 1} \end{smallmatrix}\right)$. So the result will follow from the fact that any matrix of the forms $\left(\begin{smallmatrix} 1 & 0 \\  \v^t & I_{2n - 1} \end{smallmatrix}\right), \left(\begin{smallmatrix} 1 & \v \\ 0 & I_{2n - 1} \end{smallmatrix}\right)$ can be transformed to a matrix of the form $\left(\begin{smallmatrix}1 & 0 \\  0 & \delta \end{smallmatrix}\right), \delta \in \SL_{2n - 1}(R)$ by multiplying elementary symplectic matrices on the right.
%\end{proof}
We shall now prove the main result of this section extending (\cite{cr}, Theorem 5.5).

\begin{theorem}\label{eo11}
Let $\v \in \Um_{2n}(R, I)$. Then  $ \v\E_{2n}(R, I) = \v\ESp_{2n}(R, I)$.
\end{theorem}

\begin{proof}
We may assume that $n \geq 2$ as for $n = 1$ we have $\E_{2n}(R, I) = \ESp_{2n}(R, I)$.
Clearly  $\v\ESp_{2n}(R, I) \subset \v\E_{2n}(R, I)$. So we only need to show that $\v\E_{2n}(R, I)$ is a subset of $\v\ESp_{2n}(R, I)$. Let $\w \in \v\E_{2n}(R, I)$. Then there exists $\alpha \in \E_{2n}(R, I)$ such that $\v \alpha = \w$. Let $\alpha = \prod_{i = 1}^{m} \gamma_i E_{s_i t_i}(x_i) \gamma_i^{-1}, x_i \in I, \gamma_i \in \E_{2n}(R)$. We define $A(X_1, X_2, \ldots, X_m) = \prod_{i = 1}^{m} \gamma_i E_{s_i t_i}(X_i) \gamma_i^{-1} \in \E_{2n}(R[X_1, X_2, \ldots, X_m])$. Let $X = (X_1, X_2, \ldots, X_m)$ denote the multi variable.
Now  $W(X) = \v A(X) \in \Um_{2n}(R[X])$ such that $W(x_1, x_2, \ldots, x_m) = \w$ and $W({\bf 0}) = \v$, ${\bf 0} = (0, 0, \ldots, 0)$. 

Now for any maximal ideal $\mathfrak{m}$ in $R$ we have $W(X)_{\mathfrak{m}} \sim_{\E_{2n}(R_{\mathfrak{m}}[X])} \v_{\mathfrak{m}}$ and $\v_{\mathfrak{m}} \sim_{\E_{2n}(R_{\mathfrak{m}})} e_1$.  Therefore $W(X)_{\mathfrak{m}} \in e_1\E_{2n}(R_{\mathfrak{m}}[X]) = e_1\ESp_{2n}(R_{\mathfrak{m}}[X])$ by Lemma \ref{eo10}. So we have $W(X)_{\mathfrak{m}} \sim_{\ESp_{2n}(R_{\mathfrak{m}}[X])} W({\bf 0})_{\mathfrak{m}}$. Therefore by Proposition \ref{eo8} we have  $W(X) \sim_{\ESp_{2n}(R[X])} W({\bf 0}) = \v$ i.e. there exists $B(X) \in \ESp_{2n}(R[X])$ such that $W(X)B(X) = \v$. Replacing $B(X)$ by $B(X)B({\bf 0})^{-1}$ if necessary, we may assume $B(X) \in \ESp_{2n}(R[X], (X))$. Now evaluating both sides at $X = (x_1, x_2, \ldots, x_m)$ we have $\w \in \v\ESp_{2n}(R, I)$.
\end{proof}

%It is well known due to H. Bass that if $R$ is a commutative ring of dimension $d$, $I$ is an ideal and $n \geq d + 2$, then $\E_n(R, I)$ acts transitively on $\Um_n(R, I)$. So Theorem \ref{eo11} gives the following.
%
%%\begin{corollary}
%%Let $R$ be a commutative ring of dimension $d$, $I$ an ideal and $n \geq d + 2$, then $\ESp_{2n}(R, I)$ acts transitively on $\Um_{2n}(R, I)$.
%%\end{corollary}

%%%%%%%%%%%%%%%%%%%%%%%%%%%%%%%%%%%%%%%%%%%%%%%%%%%%%%%%%%%%%%%%%%%%%%%%%
\section{improved $K_1$ stability}
Let $A$ be a non singular affine algebra of dimension $d \geq 2$ over a perfect $C_1$ field $k$. In (\cite{iis}, Theorem 1) it is shown that $\SL_n(A)/ \E_n(A) = \K_1(A)$ for $ n \geq d + 1$. Later in (\cite{sfmaa}, Corollary 7.7) this injective stability estimate is improved to $n = d$ when $k$ is algebraically closed, $d \geq 3$ and $d!k = k$. In this section we shall show that these results hold true for relative $K_1$ with respect to a principal ideal $I$ of $A$.

\begin{theorem} \label{10}
Let $A$ be an affine algebra of dimension $d \geq 2$ over a perfect $C_1$ field $k$. Then $\Um_{d+1}(A, I) = e_1\SL_{d+1}(A, I)$ for any ideal $I$ in $A$.
\end{theorem}
\begin{proof}
 Let $\v \in \Um_{d+1}(A, I)$ and  $\tilde{\v} \in \Um_{d+1}( A \oplus I, 0 \oplus I)$ be a lift of $\v$. By Proposition \ref{9a}, $A \oplus I$ is also an affine algebra of dimension $d$ over $k$. So by (\cite{iis}, Proposition 3.1) there exists  $\varepsilon \in \SL_{d+1}(A \oplus I)$ such that $\tilde{\v} \varepsilon = e_1$. Going modulo $0 \oplus I$ we have $e_ 1\overline{\varepsilon}  = e_1, \overline{\varepsilon} \in \SL_{d+1}( A)$. Now replacing $\varepsilon$ by $\varepsilon \overline{\varepsilon}^{-1}$ we may assume that $\varepsilon \in \SL_{d+1}( A \oplus I, 0 \oplus I)$ satisfying $\tilde{\v} \varepsilon = e_1$. Taking projection onto $A$ we have $\v\varepsilon = e_1$ for some $\varepsilon \in \SL_{d+1}(A, I)$ i.e. $\v \in e_1\SL_{d+1}(A, I)$.
\end{proof}

The following result is implicit in (\cite{sfmaa}, \S 7). %We detail the argument. 
 
\begin{lemma}\label{15}
Let $A$ be a non singular affine algebra of dimension $d \geq 3$ over an algebraically closed field $k$, $\frac{1}{n} \in k$. Let $\v \in \Um_d(A)$. Then $\v$ can be transformed to a  row of the form  $(w_0, w_1, w_2, \ldots, w_{d-1}^n)$ for some $\w = (w_0, w_1, w_2, \ldots, w_{d-1}) \in  \Um_d(A)$ by elementary operations.
\end{lemma}

%\begin{proof}
%We shall prove by induction on $d$. The case $d = 3$ is clear from Proposition 5.1, 6.1, Lemma 7.4  in \cite{sfmaa}. Now we assume that the statement holds true for $dim(A) = d$. Let $A$ be a nonsingular affine algebra of dimension $d +1$ and $\v= (v_0, v_1, \ldots, v_d) \in \Um_{d+1}(A)$. By Swan's version of Bertini's Theorem (\cite{cpmr}) we can add suitable multiple of $v_1, v_2, \ldots, v_d$ to $v_0$ to assume that $\overline{A} = A/v_0$ is a smooth affine algebra of dimension $d$. By induction hypothesis $(\overline{v_1}, \overline{v_2}, \ldots, \overline{v_d}) \sim_{\E_d(\overline{A})} (\overline{w_1}, \overline{w_2}, \ldots, \overline{w_d}^n)$ for some $w_1, w_2, \ldots, w_d \in A$. Therefore $\v \sim_{\E_{d+1}(A)} (v_0, w_1, w_2, \ldots, w_{d}^n)$ for $\w = (v_0, w_1, w_2,  \ldots, w_d)$ $ \in \Um_{d+1}(A)$.
%\end{proof}

\begin{lemma}\label{16}
%Let $R$ be a commutative ring, $I$ an ideal of $R$ and  
Let $\v = (v_0, v_1, v_2, \ldots, v_n) \in \Um_{n+1}(R, I)$, $n \geq 2$. Then
 $$(v_0, v_1, v_2^2, \ldots, v_n^n) \sim_{\E_{n+1}(R, I)} (v_0, v_1,\ldots, v_n^{n!}).$$
\end{lemma}
\begin{proof}

	Let $(\tilde{v_0}, \tilde{v_1}, \ldots, \tilde{v_n}) \in \Um_{n+1}(\ZZ \oplus I)$ be a lift of $\v$ under $\mathfrak{f}$ (see  Definition \ref{7a}). Then by the main theorem in introduction of \cite{V4} we have  $(\tilde{v_0}, \tilde{v_1}, \tilde{v_2}^2 \ldots, \tilde{v_n}^n) \sim_{\E_{n+1}(\ZZ \oplus I)}(\tilde{v_0}, \tilde{v_1}, \ldots, \tilde{v_n}^{n!})$. But both of $(\tilde{v_0}, \tilde{v_1}, \tilde{v_2}^2 \ldots, \tilde{v_n}^n)$, $(\tilde{v_0}, \tilde{v_1}, \ldots, \tilde{v_n}^{n!}) \equiv e_1 \vpmod {0 \oplus I}$. So by the Excision Theorem \ref{7b} we have $(\tilde{v_0}, \tilde{v_1}, \tilde{v_2}^2 \ldots, \tilde{v_n}^n) \sim_{\E_{n+1}(\ZZ \oplus I, 0 \oplus I)}(\tilde{v_0}, \tilde{v_1}, \ldots, \tilde{v_n}^{n!})$. Now the result follows taking the image under $f$ in $\Um_{n+1}(R, I)$.
\end{proof}

\begin{theorem}\label{17}
Let $A$ be a nonsingular affine algebra of dimension $d \geq 4$ over an algebraically closed field $k$, $\frac{1}{d-1!} \in k$ and $I$ an ideal of $A$. Let $\v \in \Um_d(A, I)$. Then $\v$ can be transformed to a factorial row  $\psi_{d -1!}(\w) = (w_0, w_1, w_2^2, \ldots, w_{d-1}^{d-1})$ for some $\w = (w_0, w_1, w_2, \ldots, w_{d-1}) \in  \Um_d(A, I)$ by elementary operations relative to $I$. In particular $\Um_d(A, I) = e_1\SL_d(A, I)$.
\end{theorem}

\begin{proof}
Let $\v = (v_0, v_1, \ldots, v_{d-1}) \in \Um_d(A, I)$ and $v_0 = 1- \lambda, \lambda \in I$. By Swan's version of Bertini's Theorem (\cite{cpmr}) we can add suitable multiple of $\lambda v_1, \lambda v_2, \ldots, \lambda v_d$ to $v_0$ to assume that $\overline{A} = A/(v_0)$ is a smooth affine algebra of dimension $d -1$. So by Lemma \ref{15} we have $w_1, w_2, \ldots, w_{d-1} \in  (\lambda)$  such that $(\overline{w_1}, \overline{w_2}, \ldots, \overline{w_{d-1}}) \in \Um_{d-1}(\overline{A})$ and $(\overline{v_1}, \overline{v_2},  \ldots, \overline{v_{d-1}}) \overline{\varepsilon} = (\overline{w_1}, \overline{w_2}, \ldots, \overline{w_{d-1}}^{d-1!})$  for some $\overline{\varepsilon} \in \E_{d-1}(\overline{A})$. Let $\varepsilon \in \E_{d-1}((\lambda)) \subset \E_{d-1}(I)$ be a lift of $\varepsilon$. Then we have $(v_1, v_2,  \ldots, v_{d-1}) \varepsilon \equiv (w_1, w_2, \ldots, w_{d-1}^{d-1!}) \vpmod{ Iv_0}$. So
\begin{eqnarray*}
(v_0, v_1, \ldots, v_{d-1})
& \sim_{\E_{d}(A, I)}& (v_0, w_1, w_2, \ldots, w_{d-1}^{d-1!})\\
& \sim_{\E_{d}(A, I)} &(w_0, w_1, w_2^2, \ldots, w_{d-1}^{d-1})\, ; v_0 = w_0, \text{by Lemma \ref{16}}.
\end{eqnarray*}
The last assertion is clear since  $\psi_{d -1!}(w)$ is the first row of $\beta(w, w') \in \SL_d(A, I)$ for  $w' \in \Um_d(A, I)$ satisfying $ \langle w, w' \rangle  = 1$.% (see Definition \ref{6a}).
\end{proof}

We don't know if the above theorem holds true for nonsingular affine algebras of dimension three. We ask the following question.

\begin{question}
Let $A$ be a nonsingular affine algebra of dimension three over an algebraically closed field $k$ such that $1/2 \in k$ and $I$ an ideal of $A$. Then is $\Um_3(A, I) = e_1\SL_3(A, I)$?
\end{question}

%\begin{definition}\label{18}
%A regular $k$ spot is defined to be a localization $A_\mathfrak{p}$ of an affine algebra $A$ over a field $k$ at a regular prime $\mathfrak{p} \in \mbox{Spec}(A)$ i.e. $A_\mathfrak{p}$ is a regular local ring.
%\end{definition}

The following is proved in (\cite{glgr}, Theorem 3.3).
\begin{theorem}\label{19}
Let $A$ be a regular ring essentially of finite type over a field $k$. Let $A[\underline{T}]$ denote the polynomial algebra $A[T_1, T_2, \ldots, T_n]$ and $(\underline{T})$ the ideal $(T_1, T_2, \ldots, T_n)$ in $A[T_1, T_2, \ldots, T_n]$. If $\alpha(\underline{T}) \in \GL_r(A[\underline{T}], (\underline{T}))$, then $\alpha(\underline{T}) \in \E_r(A[\underline{T}], (\underline{T}))$ for $r \geq 3$.
\end{theorem}

From now onwards by the term ``principal ideal" we shall mean an ideal generated by a single element possibly a unit.

\begin{theorem}\label{12}
Let $A$ be an affine algebra of dimension $d$ over a perfect $C_1$ field $k$, $d \geq 2$ and $I = (a)$ a principal ideal of $A$. Let $\alpha \in \SL_{d+1}(A, I)\cap \E(A, I)$. Then $\alpha$ is isotopic to identity relative to $I$. Moreover if $A$ is nonsingular then, 
\[\SL_{d+1}(A, I) \cap \E(A, I) = \E_{d+1}(A, I).\]
\end{theorem}

\begin{proof}
By classical injective stability estimate we have $\alpha \in \E_{d +2}(A, I)$. So we have an isotopy  $\gamma(T) \in \E_{d+2}(A[T], (T))$ such that $\gamma(a) = 1 \perp \alpha$ and $\gamma(0) = I_{d+2}$.  Now $e_1\gamma(T) \in \Um_{d+2}(A[T], (T^2-aT))$. So by Theorem \ref{10} the first row of $\gamma(T)$ is completable to a matrix $\gamma_1(T) \in \SL_{d+2}(A[T], (T^2-aT))$ i.e. $e_1\gamma(T) = e_1\gamma_1(T)$. Now note that $\gamma(T)\gamma_1(T)^{-1}$  is also an isotopy from $1 \perp \alpha$ to $I_{d+1}$ and it will look like
$\begin{pmatrix}
1 & 0 \\
* & \beta(T)
\end{pmatrix}$.
Clearly $\beta(T) \in \SL_{d+1}(A[T], (T))$ such that $\beta(a) = \alpha$. We define $\beta'(T) = \beta(aT)$. Then $\beta'(T) \in \SL_{d+1}(A[T], I)$ is an isotopy from $\beta'(0) = I_{d+1}$ to $\beta'(1) = \alpha$ relative to $I = (a)$.

If $A$ is nonsingular then by Theorem \ref{19} we have $\beta(T) \in \E_{d+1}(A[T], (T))$. So $\alpha = \beta(a) \in \E_{d+1}(A, I)$. Therefore $\SL_{d+1}(A, I) \cap \E(A, I) \subset \E_{d+1}(A, I)$. The reverse inclusion is obvious. 
\end{proof}

The proof of the following is verbatim that of the previous theorem. One needs to use Theorem \ref{17} in place of \ref{10}.
%Using Theorem \ref{17} we can improve Theorem \ref{12} to the following. 

\begin{theorem}\label{13}
Let $A$ be a non singular affine algebra of dimension $d$ over an algebraically closed field $k$, $d \geq 3, d! \in  k \setminus \{0\}$ and $I = (a)$ a principal ideal of $A$. Then $$\SL_{d}(A, I)\cap \E(A, I) = \E_{d}(A, I).$$
\end{theorem}
%
%\begin{proof}
%Let $\alpha \in \SL_{d}(A, I)\cap \E(A, I)$. Then by Theorem \ref{12} we have $1 \perp \alpha \in \E_{d+1}(A, I)$. Now following the proof verbatim that of previous lemma and using Theorem \ref{17} instead of Theorem \ref{10}, we have  $\beta(T) \in \SL_d(A[T], (T))$ such that $\beta(a) = \alpha$ and $\beta(0) = I_d$. 
%
%Since  $A$ is nonsingular by  Theorem \ref{19} we have $\beta(T) \in \SL_{d}(A[T], (T)) =  \E_{d}(A[T], (T))$. So $\alpha = \beta(a) \in \E_{d}(A, I)$. Therefore $\SL_{d}(A, I) \cap \E(A, I) \subset \E_{d}(A, I)$. The other inclusion is obvious. 	
%\end{proof}

%\begin{remark}\label{23}
%If one knows that Theorem \ref{17} holds true for not necessarily smooth affine algebra, then one can improve the  above lemma to say that any $\alpha \in \SL_{d}(A, I)\cap \E(A, I)$ is homotopic to identity relative to any arbitrary ideal $I$ of $A$. 
%\end{remark}

Theorems  \ref{12}, \ref{13} lead us to ask the following question.

\begin{question}\label{24}
Let $A$ be a non singular affine algebra of dimension $d$ over an algebraically closed field $k$ and $I$ an arbitrary  ideal of $A$. Then are the following true? 
\begin{itemize}

\item[1.]
$\SL_{d+1}(A, I) \cap \E(A, I) = \E_{d+1}(A, I)$ whenever $d \geq 2$. 

\item[2.] 
$\SL_{d}(A, I)\cap \E(A, I) = \E_{d}(A, I)$ whenever $d \geq 3, d! \in k \setminus 0$. 

\item[3.]
$\SL_{d -1}(A)\cap \E(A) = \E_{d -1}(A)$.

\end{itemize}
\end{question}

%%%%%%%%%%%%%%%%%%%%%%%%%%%%%%%%%%%%%%%%%%%%%%%%%%%%%%%%%%%%%%%%%%%%%%

\section{main results}
In this section we establish the main result of this article as stated in the introduction.
\begin{lemma}\label{51}
Let $A$ be a non singular affine algebra of dimension $d$ over an algebraically closed field $k$ such that $\frac{1}{d!} \in k$ and $d \equiv 2 \vpmod 4$ and $I$  a principal ideal of $A$. Let $\v \in \Um_{d}(A, I)$. Then $\psi_{d - 1!}(\v)$ can be completed to a symplectic matrix which is congruent to identity modulo $I$.
\end{lemma}

\begin{proof}
If $d = 2$ the result is obvious. So we assume that $d \geq 6$. 
We choose $\w \in \Um_{d}(A, I)$ such that $\langle \v, \w \rangle = 1$. By Lemma \ref{9} we have $\varepsilon_{J_{d - 1}} \in \ESp_{J_{d - 1}}(A, I)$ such that $S_{d - 1}(\v, \w)\varepsilon_{J_{d - 1}} =(I_{2^{d - 1} - d} \perp \gamma)\varepsilon$ for some $\gamma \in \Sp_{d}(A, I)$ and $\varepsilon \in \E_{2^{d -1}}(A, I)$. Therefore $S_{d - 1}(\v,\w)$ and $\gamma$ are stably elementary equivalent relative to $I$. Rao in (\cite{seh}, Corollary 4.3 (ii)) has shown that $S_{d -1}(\v, \w)$ and $\beta_{d - 1}(\v, \w)$ are also stably elementary equivalent relative to $I$. Therefore $\beta_{d -1}(\v, \w)\gamma^{-1} \in \SL_{d}(A, I)\cap \E(A, I) = \E_{d}(A, I)$ by Theorem \ref{13}. By Theorem \ref{eo11} $e_1\beta_{d -1}(\v, \w)\gamma^{-1} = e_1 \delta$ for some $\delta \in \ESp_{d}(A, I)$. Thus we have $\psi_{d -1!}(\v) = e_1\beta_{d -1}(\v, \w) = e_1 \delta \gamma \in e_1\Sp_{d}(A, I)$.
\end{proof}

\begin{theorem}\label{52}
Let $A$ be a non singular affine algebra of dimension $d$ over an algebraically closed field $k$ such that $d \equiv 2 \vpmod 4$, $\frac{1}{d!} \in k$ and $I$ an ideal of $A$. % Let $I$ be an ideal of $A$ such that either $I$ is principal or $A/I$ is smooth. 
Let $\v \in \Um_{d}(A, I)$. Then $\v$ can be completed to a symplectic matrix which is congruent to identity modulo $I$.
\end{theorem}

\begin{proof}
We only need to consider the case $d \geq 6$. Note that if $\v = (1 - i_1, i_2, \ldots, i_d) \in \Um_d(A, I)$ then $\v$ is $\ESp_d(A, I)$ equivalent to $(1 - i_1,  i_1i_2, i_1i_3, \ldots, i_1i_d) \in \Um_d(A, I)$. So without loss of generality we may assume that $I$ is principal.

 By  Theorem \ref{17}, $\v$ is elementarily  equivalent to $\psi_{d - 1!}(\v')$ relative to $I$ for some $\v' \in \Um_{d}(A, I)$. Using  Theorem \ref{eo11} we have $\v \varepsilon = \psi_{d -1!}(\v)$ for some $\varepsilon \in \ESp_{d}(A, I)$. Now the result follows by Lemma \ref{51}.
\end{proof}

The proof of the following will  follow verbatim that of (\cite{glgr}, Theorem 3.3) in the linear case. One may also see (\cite{isk1}, Theorem 3.8) for more details.

\begin{theorem}\label{49}
Let $A$ be a regular ring essentially of finite type over a field $k$. Then $$\Sp_{2r}(A[X], (X)) = \ESp_{2r}(A[X], (X)), r \geq 2.$$ 
\end{theorem}

\begin{theorem}\label{53}
Let $A$ be a nonsingular affine algebra of dimension $d$ over an algebraically closed field $k$ and $I = (a)$ a principal ideal. 
Assume $(d+1)! \in k^* = k \setminus \{0\}$, $d \equiv 1 \vpmod 4$.% and $d \geq 5$. 
Then
$$\Sp_{d-1}(A, I) \cap \ESp_{d+1}(A, I) = \ESp_{d-1}(A, I).$$
\end{theorem}

\begin{proof} 

Let $ \sigma \in \Sp_{d-1}(A, I) \cap \ESp_{d+1}(A, I)$. Let $\delta(T) \in \ESp_{d + 1}(A[T])$ be an elementary symplectic isotopy between  $\delta(a) = I_2 \perp \sigma$ and $\delta(0) = I_{d + 1}$. Then $\v(T) = e_1\delta(T) \in \Um_{d+1}(A[T], T^2 - aT)$. By previous Theorem \ref{52}, $\v(T)$ can be completed to a symplectic matrix $\alpha(T)$ which is congruent to identity modulo $(T^2 - aT)$. So we have $e_1\delta(T) = e_1 \alpha(T)$. Now $\delta(T)\alpha(T)^{-1}$ is also an isotopy between $I_{d + 1}$ and $I_2 \perp \sigma$. We have
$$\delta(T)\alpha(T)^{-1} = 
\begin{pmatrix}
1 & 0 & 0\\
* & 1 & *\\
* & 0 & \eta(T)
\end{pmatrix}$$
for some $\eta(T) \in \Sp_{d-1}(A[T], (T))$. Clearly $\eta(T)$ is a symplectic isotopy between $\eta(0) = I_{d-1}$ and $\eta(a) = \sigma$. Since $A$ is nonsingular %$\Sp_{2r}(A[T], (T)) = \ESp_{2r}(A[T], (T)), r \geq 2$ 
by Theorem \ref{49} %. Consequently
 $\eta(T) \in \ESp_{d-1}(A[T], (T))$. So $\sigma = \eta(a) \in \ESp_{d-1}(A, I)$. Therefore $\Sp_{d-1}(A, I) \cap \ESp_{d+1}(A, I) \subset \ESp_{d-1}(A, I)$. The reverse inclusion is trivial.
\end{proof}
 
Theorems \ref{52} and \ref{53} lead us to ask the following question.

\begin{question}\label{54}
Let $A$ be a non singular affine algebra of dimension $d$ over an algebraically closed field $k$ and $I$ an ideal of $A$. Then are the following true in general?
\begin{itemize}
\item[1]
If $d! \in k \setminus \{0\}$, then $\Um_d(A, I) = e_1\Sp_d(A, I)$.
\item[2]
If $(d+1)! \in  k \setminus \{0\}$, then $\Sp_{d-1}(A, I) \cap \ESp_{d+1}(A, I) = \ESp_{d-1}(A, I)$.
\end{itemize}
%By similar arguments one can improve the main theorem of \cite{bcr} as follows.
\end{question}
%\begin{theorem}\label{55}
%Let $A$ be a nonsingular affine algebra of dimension $d$ over an algebraically closed field $k$ and $I$ a principal ideal. Assume $(d+1)! \in k^*, d \equiv 0 \vpmod 4$. Then $$\Sp_{d}(A, I) \cap \ESp_{d+2}(A, I) = \ESp_{d}(A, I).$$
%\end{theorem}

%%%%%%%%%%%%%%%%%%%%%%%%%%%%%%%%%%%%%%%%%%%%%%%%%%%%%%%%%%%%%%%%%%%%%%%%%%

\section{an example}
Let $A$ be an affine algebra over a field $k$. Let $\v = (v_0, v_1, \ldots, v_n), \w = (w_0, w_1, \ldots, w_n)$ $\in \Um_{n+1}(A)$, $n = 2k +1 \geq 3$ be such that $\sum_{i = 0}^n v_iw_i = 1$. We can complete $\v$ to an elementary matrix in $\E_{n+1}(A_{v_0})$ and therefore to a matrix in $\ESp_{n+1}(A_{v_0})$ by Theorem \ref{eo11}. Let $\varepsilon_1 \in \ESp_{n+1}(A_{v_0})$ be such that $e_1 \varepsilon_1 =  \v$.  Since $ \v \w^t = 1$, we have $\varepsilon_1 \w^t = (1, *, *, \ldots, *)^t$. So we can suitably modify $\varepsilon_1$ by elementary symplectic action (multiplication by $se_{i1}(*) , 1 < i$ from the left) retaining its first row and further assume that $\varepsilon_1 \w^t = e_1^t$.

Since $v_0$ is in the Jacobson radical of $A_{1 + v_0A}$, we can also complete $\v$ to an elementary matrix in $\E_{n+1}(A_{1 + v_0A})$. By  Theorem \ref{eo11} we shall have $\varepsilon_2 \in \ESp_{n+1}(A_{1 + v_0A})$ such that $e_1 \varepsilon_2 =  \v$. After a suitable modification, if necessary, as described earlier we  assume that $\varepsilon_2 \w^t = e_1^t$.

Now $\varepsilon_1 \varepsilon_2^{-1} = \ESp_{n+1}(A_{v_0(1 + v_0A)})$ such that $e_1\varepsilon_1 \varepsilon_2^{-1}= e_1$ and $\varepsilon_1 \varepsilon_2^{-1}e_1^t = e_1^t$. So by Lemma \ref{rcm}  $\varepsilon_1 \varepsilon_2^{-1}$ will look like $I_2 \perp \gamma$ for some $\gamma \in \Sp_{n-1}(A_{v_0(1 + v_0A)}) \cap \ESp_{n+1}(A_{v_0(1 + v_0A)})$.

If $\v$ is not completable then $\gamma$ must not split into two matrices over $A_{v_0}$ and $A_{(1 + v_0A)}$. Therefore $\gamma$ is not contained in $\E_{n-1}(A_{v_0(1 + v_0A)})$. The following will follow from (\cite{sfmk}, Claim 4, 3rd paragraph of page no 1443).

\begin{proposition}\label{61} Let $p$ be any prime number. Then the following assertions hold true.
\begin{enumerate} 
\item
There exists an affine domain $A$  of dimension $p+1$ over a $C_1$ field $k$  and  $\v = (v_0, v_1, \ldots, v_p) \in \Um_{p+1}(A)$ which is not completable.
\item
There exists an affine domain of dimension $p+2$ over an algebraically closed field $k$ and a unimodular row $\v = (v_0, v_1, \ldots, v_p) \in \Um_{p+1}(A)$ which is not completable.  
\end{enumerate}
\end{proposition}

  We shall choose $\v \in \Um_{p+1}(A)$ given by the first assertion and assume $n = p \geq 3$ in the above discussion. Let $B$ be obtained from $A$ by inverting all non zero polynomials in $k[v_0]$. Then $B$ is an affine domain of dimension $p$ over the $C_2$ field $k(v_0)$ and $A_{v_0(1 + v_0A)}$ is a further localization of $B$. So we have  $g \in B$ such that $\gamma \in \Sp_{p-1}(B_g) \cap \ESp_{p+1}(B_g)$. Note that $B_g$ is also an affine domain of dimension $p$ over the $C_2$ field $k(v_0)$. Also $\gamma$ is not even elementary as the unimodular row  is not completable.  Therefore Theorem \ref{53} is not true if the ground field is a $C_2$ field.

Now we shall choose the $\v \in \Um_{p+1}(A)$ given by the second assertion and $n=p \geq 3, p = 3 \vpmod 4$ in the  discussion preceding Proposition \ref{61}. We  have $\gamma \in \Sp_{p-1}(A_g) \cap \ESp_{p+1}(A_g)$ but $\gamma \not \in \ESp_{p-1}(A_g)$ for some $g \in A$. Note that $A_g$ is an affine domain of dimension $d = p +2, d = 1 \vpmod 4$. So $\gamma \in \Sp_{d-3}(A_g) \cap \ESp_{d-1}(A_g)$ but $\gamma \not \in \ESp_{d-3}(A_g)$. Therefore our estimate in Theorem \ref{53} is the best possible.

%I would like to conclude with the following questions which are not known to us so far.
%\begin{question}\label{62}æ
%Let $A$ be a smooth affine algebra of dimension $d$ over an algebraically closed field $d$.
%\begin{enumerate}
%\item 
%We don't know if Theorem \ref{55} gives  the best possible estimate. 
%\item
%If $d \equiv 2 \vpmod 4$ or $d  \equiv 3 \vpmod 4$ then what is the least even integer $k$ such that  $$\Sp_{k}(A) \cap \ESp_{k+2}(A) = \ESp_{k}(A)?$$
%
%\end{enumerate}
%
%\end{question}

\begin{acknowledgement}
This paper is a part of my PhD thesis. I am sincerely grateful to my advisor Ravi A.Rao for asking me the above problem and answering patiently my innumerable questions raised in my mind during the preparation of this article. I am also thankful to SPM Fellowship, CSIR,  (SPM - 07 - /858(0051)/ 2008 - EMR - 1) for the financial support.  

\end{acknowledgement}

\end{document}